\documentclass[11pt]{amsart}
\usepackage{amssymb}
\usepackage{amsmath,amssymb,amsfonts,amsthm,
latexsym, amscd, epsfig,color}
\usepackage{mathrsfs}
\usepackage{eucal}
\usepackage{xspace}
\usepackage{a4wide}
\usepackage{colordvi}

\usepackage{cite}

\usepackage{xcolor}
\usepackage{hyperref}
\definecolor{darkgreen}{rgb}{0,0.4,0}
\definecolor{BrickRed}{rgb}{0.65,0.08,0}
\hypersetup{colorlinks=true,linkcolor=blue,citecolor=red,filecolor=BrickRed,urlcolor=darkgreen}

\newcommand{\polya}{P{\'o}lya~}
\newcommand{\PR}{\mathbb{P}}
\newcommand{\LandauO}{\CMcal{O}}

\newcommand{\E}{\mathbb{E}}

\newcommand{\Tc}{\CMcal{T}}

\newcommand{\N}{\mathbb{N}}

\newcommand{\R}{{\mathbb R}}

\newcommand{\Aut}{\operatorname{Aut}}

\newcounter{mycount}
\theoremstyle{plain}
\newtheorem{theorem}[mycount]{Theorem}
\newtheorem{corollary}[mycount]{Corollary}

\newtheorem{lemma}[mycount]{Lemma}
\newtheorem{proposition}[mycount]{Proposition}
\theoremstyle{remark}
\newtheorem{remark}{Remark}
\theoremstyle{definition}
\newtheorem{definition}{Definition}

\theoremstyle{remark}
\newtheorem{example}{Example}
\numberwithin{equation}{section} \numberwithin{figure}{section}

\begin{document}
\title{On the shape of random P\'olya structures}
\author{Bernhard Gittenberger, Emma Yu Jin and Michael Wallner}
\thanks{Corresponding author email: gittenberger@dmg.tuwien.ac.at. This work was supported by the Austrian Research Fund (FWF), 
grant SFB F50-03. }
\thanks{A preliminary version of this paper was published in the Proceedings of ANALCO 2017.}
\address{Institut f\"{u}r Diskrete Mathematik und Geometrie, Technische Universit\"{a}t Wien, Wiedner Hauptstr. 8--10/104, 1040 Vienna, Austria}
\email{gittenberger@dmg.tuwien.ac.at}
\email{yu.jin@tuwien.ac.at}
\email{michael.wallner@tuwien.ac.at}

\begin{abstract}
Panagiotou and Stufler recently proved an important fact on their way to establish the scaling limits of random \polya trees:
a uniform random \polya tree of size $n$ consists of a conditioned critical Galton-Watson tree $C_n$ and many small forests, where with probability tending to one, as $n$ tends to infinity, any forest $F_n(v)$, that is attached to a node $v$ in $C_n$, is maximally of size $\vert F_n(v)\vert=O(\log n)$.
Their proof used the framework of a Boltzmann sampler and deviation inequalities.

In this paper, first, we employ a unified framework in analytic combinatorics to prove this fact with additional improvements for $\vert F_n(v)\vert$, namely $\vert F_n(v)\vert=\Theta(\log n)$.
Second, we give a combinatorial interpretation of the rational weights of these forests and the defining substitution process
in terms of automorphisms associated to a given \polya tree.
Third, we derive the limit probability that for a random node $v$ the attached forest $F_n(v)$ is of a given size. Moreover, structural properties of those forests like the number of their components are studied. Finally, we extend all results to other P\'{o}lya structures.
\end{abstract}



\maketitle

\section{Introduction}\label{sec:intro}
In this section we first recall the asymptotic estimation of the number of \polya trees with $n$ nodes from the literature \cite{niwi78,otte48,poly37}.
Then, we briefly discuss simply generated trees and give finally an outline of the paper. 

\subsection{\polya trees}
A {\em \polya tree} is a rooted unlabelled non-plane tree.
The {\em size} of a tree is given by the number of its nodes. We denote by $t_{n}$ the number of \polya trees of size $n$ and by $T(z) = \sum_{n \geq 1} t_n z^n$ the corresponding ordinary generating function.  By P\'{o}lya's enumeration theory \cite{poly37}, the generating function $T(z)$ satisfies
\begin{align}\label{E:penum1}
T(z)=z\exp\left(\sum_{i=1}^{\infty}\frac{T(z^i)}{i}\right).
\end{align}
The first few terms of $T(z)$ are then
\begin{align}\label{eq:polyaimplicit}
T(z)=z + z^2 + 2 z^3 + 4 z^4 + 9 z^5 + 20 z^6 + 48 z^7 + 115 z^8 + 286 z^9 + 719 z^{10} + \cdots,
\end{align}
(see OEIS~A$000081$, \cite{Sloane}). By differentiating both sides of (\ref{E:penum1}) with respect to $z$,  one can derive a recurrence relation of $t_n$ (see \cite[Chapter~29]{niwi78} and \cite{otte48}), which is
\begin{align*}
t_n=\frac{1}{n-1}\sum_{i=1}^{n-1}t_{n-i}\sum_{m\vert i}mt_m, \quad \mbox{ for }\, n>1,\,\, \mbox{ and }\, t_1=1.
\end{align*}
\polya \cite{poly37} showed that the radius of convergence $\rho$ of $T(z)$ satisfies $0 < \rho < 1$ and that $\rho$ is the unique singularity on the circle of convergence. Subsequently, Otter \cite{otte48} proved that $T(\rho)=1$ as well as the singular expansion
\begin{align}
	\label{eq:polyaasympt}
	T(z) &= 1 - b\left(\rho-z\right)^{1/2} + c(\rho-z) + \CMcal{O}\left((\rho-z)^{3/2}\right),
\end{align}
where $\rho \approx 0.3383219$, $b \approx 2.68112$ and $c = b^2/3 \approx 2.39614$.
Moreover, he derived
\begin{align*}
	t_n &= \frac{b \sqrt{\rho}}{2 \sqrt{\pi}} \frac{\rho^{-n}}{\sqrt{n^{3}}} \left(1 + \CMcal{O}\left(\frac{1}{n}\right)\right).
\end{align*}

\subsection{Relation between \polya and simply generated trees}
We will see that $T(z)$ is connected with the \emph{exponential generating function} of Cayley
trees, which are rooted labelled trees. ``With a minor abuse of notation'' (\emph{cf.} \cite[Ex.~10.2]{Janson12_SGT}), Cayley trees belong to the class of \emph{simply generated trees}. Simply generated trees have been introduced by Meir and Moon \cite{memo:78} to describe a weighted version of rooted trees. They are defined by the functional equation
\begin{align}\label{E:sgt}
	y(z) &= z \Phi(y(z)),\qquad \mbox{ with }\quad\Phi(z)=\sum_{j\ge 0}\phi_j\,z^j, \quad \phi_j \geq 0.
\end{align}
The power series $y(x)=\sum_{n\ge 1}y_nx^n$ has nonnegative coefficients and is the generating
function of {\em weighted} simply generated trees. One usually assumes that $\phi_0>0$ and
$\phi_j>0$ for some $j\ge 2$ to exclude the trivial cases. In particular, in the above-mentioned
sense, {\em Cayley trees} can be seen as simply generated trees which are characterized by $\Phi(z) = \exp(z)$.
It is well known that the number of rooted Cayley trees of size $n$ is given by $n^{n-1}$.
Let
\begin{equation} \label{fctnC}
C(z) = \sum_{n \geq 0} n^{n-1} \frac{z^n}{n!},
\end{equation}
be the associated exponential generating function.
Then, by construction it satisfies $C(z)=z\exp(C(z))$. In contrast, \polya trees are not simply generated (see \cite{DG10} for a simple proof of this fact).
Note that though $T(z)$ and $C(z)$ are closely related, \polya trees are not related to Cayley
trees in a strict sense, but rather to a certain class of weighted {\em unlabeled non-plane} trees, which will be called
$C$-trees in the sequel and have the {\em ordinary} generating function $C(z)$. 

Informally speaking, a \polya tree is constructed from a $C$-tree where to each node a forest is attached. The family where the forests are taken from will be called $D$-forests. So, a \polya tree is (as above ``with a minor abuse of notation'') a simply generated tree with small decorations on each of its vertices. This follows from a limit theorem by Panagiotou and Stufler \cite{PS} where trees are seen as random metric spaces which converge to a limit space, the so-called scaling limit, when suitably rescaled. 
Their proof uses advanced probabilistic methods. A goal of this paper is to understand this limit from an analytic combinatorics \cite{FS} point of view and to offer a somewhat more elementary approach to this limit theorem. We will not reprove the complete limit theorem, but we set up a description in terms of generating functions which exhibits combinatorially the above mentioned relation between \polya trees and simply generated trees. Moreover, this description allows a detailed analysis of the decorations and leads to some extensions of Panagiotou and Stufler's \cite{PS} results on the decorations. 

\subsection{Outline of the paper}
The paper is organized as follows. In Section~\ref{basics} we present the combinatorial setup and the main results. Section~\ref{S:D-forest} is devoted to the study of the size of $D$-forests and the size of the $C$-tree $C_n$ in a random \polya tree $T_n$. We also offer new proofs of some results from \cite{PS} using the analytic combinatorics framework. In Section~\ref{S:size-D} we prove Theorems~\ref{T:2} and \ref{T:312}. A detailed study of $D$-forest is the topic of Section~\ref{S:size-D}. There we study the distribution of the size a randomly chosen $D$-forest within a \polya tree as well as the distribution of the number of components.  
All the results can be generalized to further \polya structures, albeit we do not obtain as explicit expressions as in the case of \polya trees. This will be the topic of Section~\ref{other}. We conclude in Section~\ref{S:last} with some final remarks.

\section{Basic structures and main results}
\label{basics}

The generating function $C(z)$ defined in \eqref{fctnC} counts several combinatorial objects. In this section we comment on the different interpretations and answer the question of what $C$-trees really are.

For a set $\Tc$ we define two functions: a size function $|\cdot| : \Tc \to \N$ (normally the number of its nodes) and a weight function $w : \Tc \to \R$. We call $\Tc$ a combinatorial class if the number of elements $T \in \Tc$ of any given size is finite.
The (weighted) generating function $T(z)$ of $\Tc$ is given by
$$T(z) = \sum_{T \in \Tc} w(T) z^{|T|},$$
and the (weighted) exponential generating function $\hat{T}(z)$ of $\Tc$ is given by
$$\hat{T}(z) = \sum_{T \in \Tc} w(T) \frac{z^{|T|}}{|T|!} = \sum_{n \geq 0} \Big(\sum_{\substack{T \in \Tc \\ |T| = n}} w(T) \Big) \frac{z^n}{n!}.$$

First, let $\CMcal{C}_1$ be the combinatorial class of Cayley trees. These are labeled trees with the constant weight function $w(T) = 1$ for all $T \in \CMcal{C}_1$, see~\cite{poly37}. Then, $C(z)$ is the corresponding exponential generating function.

Second, let $\CMcal{C}_2$ be the combinatorial class of simply generated trees with $\Phi(z) = \exp(z)$. These are unlabeled plane trees with weight function
\begin{align}\label{E:simple}
w(T) = \prod_{k \geq 0} \left(\frac{1}{k!}\right)^{n_k(T)}
\end{align}
where $n_k(T)$ is the number of nodes of $T$ with outdegree $k$, see~\cite{Janson12_SGT,memo:78}. Then, $C(z)$ is the generating function of these trees.

Third, we can define the class $\CMcal{C}_3$ of $C$-trees after all. 

\begin{definition}
A $C$-tree is a rooted non-plane tree $T$ with weight 
\begin{align}
w(T)=\operatorname{e}(T)\prod_{k \geq 0} \left(\frac{1}{k!}\right)^{n_k(T)}
\end{align}
where $\operatorname{e}(T)$ is the number of embeddings of $T$ into the plane. 

The class of all $C$-trees is called $\CMcal{C}_3$. 
\end{definition}

\begin{remark}
The class $\CMcal{C}_3$ is the non-plane version of the class $\CMcal{C}_2$. 
\end{remark}

\begin{lemma}
The generating function of $\CMcal{C}_3$ is $C(z)$. 
\end{lemma}

\begin{proof}
For a given $C$-tree $T$ of size $n$ and with root of outdegree $d$, let $T_1,T_2,\ldots,T_k$ be the distinct subtrees of the root, appearing with multiplicities $m_1,\dots,m_k$, respectively. Then, $d=m_1+m_2+\cdots+m_k$ with $m_i \geq 1$. We define 
\begin{align*}
	\delta(T):= | \{ \text{all permutations of }(\underbrace{T_1,\ldots,T_1}_{m_1},\underbrace{T_2,\ldots,T_2}_{m_2},\ldots,
\underbrace{T_k,\ldots,T_k}_{m_k}) \} |=\binom{d}{m_1,m_2,\ldots,m_k}.
\end{align*}
Hence, we get
\begin{align*}
	w(T) = \frac{\delta(T)}{d!} \prod_{i=1}^k w(T_i)^{m_i}
=\left(\frac{w(T_1)^{m_1}}{m_1!}\right)\left(\frac{w(T_2)^{m_2}}{m_2!}\right)\cdots \left(\frac{w(T_k)^{m_k}}{m_k!}\right).
\end{align*}
This implies that the generating function of the class $\CMcal{C}_3$ satisfies $C(z)=z \exp(C(z))$.
\end{proof}

Now, we turn to the introduction of $D$-forests. We begin with some preliminary observations: 
In order to analyze the dominant singularity of $T(z)$, we follow \cite{otte48, poly37}, see also \cite[Chapter~VII.5]{FS}, and we rewrite
\eqref{eq:polyaimplicit} into
\begin{align}
	\label{eq:polyadecoA}
	T(z) &= z e^{T(z)} D(z),\quad\,\mbox{ where }\,\quad D(z)=\sum_{n \geq 0 } d_n z^n = \exp\left(\sum_{i=2}^{\infty}\frac{T(z^i)}{i}\right).
\end{align}
We observe that $D(z)$ is analytic for $|z| < \sqrt{\rho}<1$ and that $\sqrt{\rho}>\rho$.
From~\eqref{eq:polyadecoA} it follows that $T(z)$ can be expressed in terms of the generating function of $C$-trees: Indeed, assume that $T(z)$ is a function $H(zD(z))$ depending on $zD(z)$.
By~\eqref{eq:polyadecoA} this is equivalent to $H(x) = x \exp(H(x))$. Yet, this is the functional
equation for the generating function of $C$-trees. As this functional equation has a unique power series solution we have $H(x) = C(x)$, and we just proved
\begin{align}
	\label{eq:polyadeco}
	T(z) &= C(z D(z) ).
\end{align}
Note that $T(z)=C(zD(z))$ is a case of a super-critical composition schema which is characterized by the fact that the dominant singularity of $T(z)$ is strictly smaller than that of~$D(z)$. In other words, the dominant singularity $\rho$ of $T(z)$ is determined by the outer function~$C(z)$. Indeed, $\rho\,D(\rho)=e^{-1}$, because $e^{-1}$ is the unique dominant singularity of~$C(z)$. Moreover, the composition schema corresponds to the substitution construction of combinatorial structures \cite{FS}. Indeed, \eqref{eq:polyadeco} shows that a \polya tree is a $C$-tree where to each vertex there has been attached a combinatorial object from a combinatorial class associated with the generating function $D(z)$. As \polya trees are trees, the $D$-structures must be forests. Inspecting the generating function $D(z)$ more closely will show that there is a natural way of defining these $D$-forests. 

\begin{definition}
Let $\operatorname{MSET}^{(\geq 2)}(\CMcal{T})$ denote the class of all multisets (or forests) of \polya trees where each of its distinct components appears at least in duplicate (see Figure~\ref{fig:smalldforests}). For $F\in\operatorname{MSET}^{(\geq 2)}(\CMcal{T})$ let $\Aut(F)$ denote the automorphism group of $F$. Moreover, let $\sigma_i$ denote the number of cycles of length $i$ in the automorphism $\sigma\in\Aut(F)$. To each $F\in\operatorname{MSET}^{(\geq 2)}(\CMcal{T})$ we assign the weight 
$$
w(F)=\frac{|\{ \sigma \in \Aut(F) ~|~ \sigma_1 = 0\}|}{|\Aut(F)|}.
$$
Then the class $(\operatorname{MSET}^{(\geq 2)}(\CMcal{T}),w)$ is called the class of $D$-forests.
\end{definition}

Our first main result is that this is indeed the combinatorially natural weighting for the $D$-forests satisfying \eqref{eq:polyadeco} as well as relating the weights of $C$-trees in terms of automorphisms associated to a given \polya tree. In particular, we will show that the cumulative weight $d_n$ (defined in \eqref{E:D}) of all such forests of size $n$ satisfies
\begin{align*}
	d_n = \sum_{\substack{F \in \operatorname{MSET}^{(\geq 2)}( \CMcal{T} )\\|F| = n}} \frac{|\{ \sigma \in \Aut(F) ~|~ \sigma_1 = 0\}|}{|\Aut(F)|}
\end{align*}
From (\ref{eq:polyaimplicit}) and (\ref{eq:polyadecoA}) one gets the first values of this sequence:
\begin{align}\label{E:D}
	D(z)=\sum_{n=0}^{\infty}d_n z^n= 1 + \frac{1}{2}z^2 + \frac{1}{3}z^3 + \frac{7}{8}z^4 + \frac{11}{30}z^5 + \frac{281}{144}z^6 +\frac{449}{840}z^7 + \cdots.
\end{align}
From \eqref{eq:polyadecoA} we can derive a recursion of $d_n$ as well. We get
\begin{align*}
d_n=\frac{1}{n}\sum_{i=2}^{n}d_{n-i}\sum_{\substack{m\vert i\\m\ne i}}mt_m, \quad\mbox{ for }\,\,n\ge 2,
\end{align*}
as well as $d_0=1$, and $d_1=0$.

Before we can formulate the first main result, we have to introduce further generating functions. 
Let $c_{n,k}$ denote the cumulative weight of all $C$-trees of size $k$ that are contained in
\polya trees of size $n$. By $t_{c,n}(u)$ and $T_c(z,u)$ we denote the corresponding generating
function and the bivariate generating function of $\left(c_{n,k}\right)_{n,k\ge 0}$, respectively, that is,
\begin{align*}
	t_{c,n}(u) &= \sum_{k=1}^n c_{n,k}u^k \quad \mbox{ and }\quad
T_c(z,u)=\sum_{n\ge 0}t_{c,n}(u)z^n.
\end{align*}
Note that $c_{n,k}$ is in general not an integer. By marking the nodes of all $C$-trees in \polya trees we find a functional equation for the bivariate generating function ${T}_c(z,u)$, which is
\begin{eqnarray}\label{E:bi1}
{T}_c(z,u)=zu\exp\left({T}_c(z,u)\right)\exp\left(\sum_{i=2}^{\infty}\frac{{T}(z^i)}{i}\right)
=zu\exp\left({T}_c(z,u)\right)D(z).
\end{eqnarray}
\begin{figure}
	\begin{center}
	\includegraphics[width=0.6\textwidth]{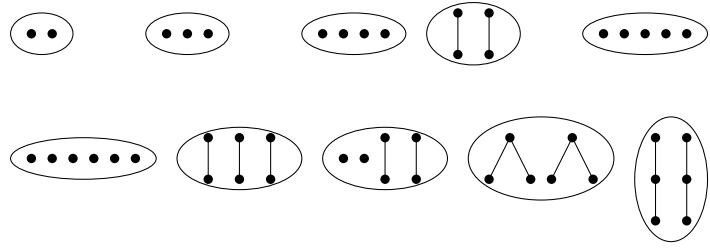}
	\end{center}
	\caption{All $D$-forests of size $2,3,4,5,6$.}
	\label{fig:smalldforests}
\end{figure}
Now we are ready to state the first main result: 
\begin{theorem}\label{T:2}
Let $\CMcal{T}$ be the set of all \polya trees, and $\operatorname{MSET}^{(\geq 2)}(\CMcal{T})$ be the multiset (or forest) of \polya trees where each tree appears at least twice if it appears at all.  
Then the cumulative weight $d_n$ (defined in \eqref{E:D}) of all such forests of size $n$ satisfies
\begin{align*}
	d_n = \sum_{\substack{F \in \operatorname{MSET}^{(\geq 2)}( \CMcal{T} )\\|F| = n}} \frac{|\{ \sigma \in \Aut(F) ~|~ \sigma_1 = 0\}|}{|\Aut(F)|}
\end{align*}
where $\Aut(F)$ is the automorphism group of $F$. Furthermore, the polynomial associated to $C$-trees in \polya trees of size $n$ is given by
\begin{align*}
	t_{c,n}(u)=\sum_{
	T \in \CMcal{T}, \; \vert T\vert=n} t_{T}(u), \quad \mbox{ where }\,\,t_{T}(u)=\frac{1}{|\Aut(T)|} \sum_{\sigma \in \Aut(T)} u^{\sigma_1}.
\end{align*}
In particular, for all $T\in \CMcal{T}$, we have $t_T'(1)=|\CMcal{P}(T)|$
where $\CMcal{P}(T)$ is the set of all trees which are obtained by pointing (or coloring) one single node in $T$.
\end{theorem}

Note that the decomposition of a \polya tree into a $C$-tree and $D$-forests is in general not unique, 
because the $D$-forests consist of \polya trees and their vertices are not distinguishable from those of the $C$-tree, see Figure~\ref{F:0}. 
%
\begin{figure}[htbp]
\begin{center}
\includegraphics[scale=1.2]{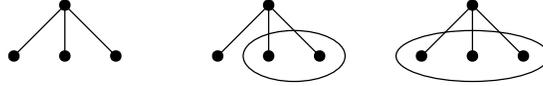}
\caption{The decomposition of a \polya tree with four nodes into a $C$-tree (non-circled nodes)
and $D$-forests (circled nodes). For this \polya tree there are three different decompositions.\label{F:0}}
\end{center}
\end{figure}
However, for a given \polya tree $T$ the polynomial $t_{T}(u)$ gives rise to a probabilistic interpretation of the composition scheme~\eqref{eq:polyadeco} (see also Example~\ref{ex:probinterpretation} in Section~\ref{S:size-D}): Let $T_n$ denote a uniform random \polya tree of size $n$. Then select one automorphism $\sigma$ of $T_n$ uniformly at random and let $C_n$ be the (random) subtree of $T_n$ consisting of all fixed points of $\sigma$. The tree $C_n$ is indeed a $C$-tree, 
\emph{i.e.}, the remaining vertices in $T_n$ form a set of $D$-forests. So, fixing a \polya tree together with one of its automorphisms uniquely defines a decomposition into a $C$-tree and a set of 
$D$-forests. 

The coefficient of $u^k$ can then be interpreted as the probability that the underlying $C$-tree is of size $k$. In other words, $t_{T}(u)$ is the probability generating function of the random variable $C_T$ of the number of $C$-tree nodes in the tree $T$ defined by
\begin{align}
	\label{eq:probCT}
	\PR(C_T = k) := [u^k] t_{T}(u).
\end{align}
The random variable $C_T$ is a refinement of $T_n$ in the following sense: 
\begin{align*}
	\PR(C_T = k) = \PR\left(|C_n| = k ~|~ T_n = T\right)
\end{align*}

\begin{figure}[htbp]
\begin{center}
\includegraphics[scale=0.7]{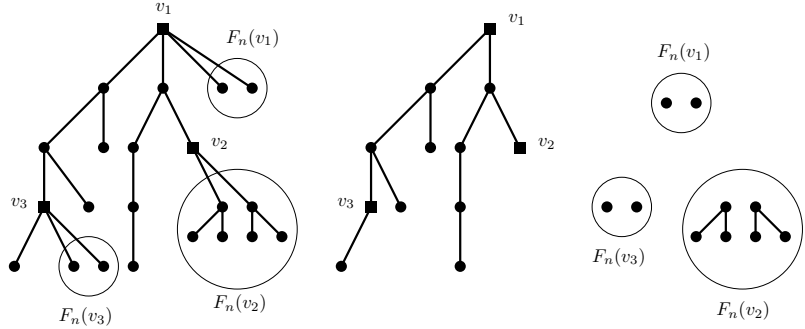}
\caption{A random \polya tree $T_n$ (left), a (possible) $C$-tree $C_n$ (middle) that is contained in $T_n$ where all $D$-forests $F_n(v)$, except $F_n(v_1),F_n(v_2),F_n(v_3)$ (right), are empty.\label{F:1}}
\end{center}
\end{figure}

Now, let us turn to the second main result. Select a random \polya tree $T_n$ and one of its
automorphisms and consider the random $C$-tree $C_n$ given by this choice. For every vertex
$v$ of $C_n$, we use $F_n(v)$ to denote the $D$-forest that is attached to the vertex $v$ in $T_n$, see Figure~\ref{F:1}.

Let $L_n$ be the maximal size of a $D$-forest contained in $T_n$, that is, $\vert F_n(v)\vert \le L_n$ holds for all $v$ of $C_n$ and the inequality is sharp. 
\begin{theorem}\label{T:1}
For $0<s<1$,
\begin{eqnarray}\label{E:mainone}
(1-(\log n)^{-s})\left(\frac{-2\log n}{\log\rho}\right)\le L_n \le (1+(\log n)^{-s})\left(\frac{-2\log n}{\log\rho}\right)
\end{eqnarray}
holds with probability $1-o(1)$.
\end{theorem}
Our second main result is a proof of Theorem~\ref{T:1} by applying the unified framework of Gourdon \cite{Gourdon}. A big-$O$ result for the upper bound was given by Panagiotou and Stufler~\cite[Eq.~(5.5)]{PS}.

Finally, we derive the limiting probability that for a random node $v$ the attached forest $F_n(v)$ is of a given size. This result is consistent with the Boltzmann sampler from \cite{PS}. The precise statement of our third main result is the following:
\begin{theorem}\label{T:312}
The generating function $T^{[m]}(z,u)$ of \polya trees, where each vertex is marked by $z$, and each weighted $D$-forest of size $m$ is marked by $u$, is given by
	\begin{align}\label{E:bi2}
		T^{[m]}(z,u) &= C \left( u z d_m z^m + z \left( D(z) - d_m z^m \right) \right),
	\end{align}
where $d_m=[z^m]D(z)$.
The probability that the $D$-forest $F_n(v)$ attached to a random $C$-tree node $v$ is of size $m$ is given by
	\begin{align*}
		\PR\left( \vert F_n(v)\vert = m \right) = \frac{d_m \rho^m}{D(\rho)} \left( 1 + \LandauO\left(n^{-1}\right)\right).
	\end{align*}
\end{theorem}


\section{The maximal size of a \texorpdfstring{$D$}{D}-forest}
\label{S:D-forest}
We will use the generating function approach from \cite{Gourdon} to analyze the maximal size $L_n$ of $D$-forests in a random \polya tree $T_n$, which provides a new proof of Theorem~\ref{T:1}. Following the same approach, we can establish a central limit theorem for the random variable $\vert C_n\vert$, which has been done in \cite{Ben2} for the more general random $\CMcal{R}$-enriched trees.

{\em Proof of Theorem~\ref{T:1}}.
In (5.5) of \cite{PS}, only an upper bound of $L_n$ is given. By directly applying Gourdon's results (Theorem~4 and Corollary $3$ of \cite{Gourdon}) for the super-critical composition schema, we find that for any positive $m$,
\begin{align*}
\mathbb{P}[L_n\le m]&=\exp\left(-\frac{c_1n}{m^{3/2}}\rho^{m/2}\right)(1+\CMcal{O}(\exp(-m\varepsilon))),
\end{align*}
where 
\begin{align*} 
c_1 &\sim \frac{b}{2\sqrt{\pi}(1-\sqrt{\rho})(D(\rho)+\rho D'(\rho))},
\end{align*}
as $n\to \infty$.
Moreover, the maximal size $L_n$ satisfies asymptotically, as $n\rightarrow \infty$,
\begin{align*}
\mathbb{E}L_n=-\frac{2\log\,n}{\log \rho}-\frac{3}{2}\frac{2}{\log \rho}\log\log\,n+\CMcal{O}(1)\quad\mbox{ and }\quad
\mathbb{V}\mbox{ar}\,L_n=\CMcal{O}(1).
\end{align*}
By using Chebyshev's inequality, one can prove that $L_n$ is highly concentrated around the mean~$\mathbb{E}L_n$. We set $\varepsilon_n=(\log n)^{-s}$ where $0<s<1$ and we get
\begin{eqnarray*}
\mathbb{P}(\vert L_n-\mathbb{E}L_n\vert\ge \varepsilon_n\cdot\mathbb{E}L_n)
\le\frac{\mathbb{V}\mbox{ar}\,L_n}{\varepsilon_n^2\cdot(\mathbb{E}L_n)^2}=o(1),
\end{eqnarray*}
which means that (\ref{E:mainone}) holds with probability $1-o(1)$.
\qed

It was shown in \cite{Ben2} that the size $\vert C_n\vert$ of the $C$-tree $C_n$ in $T_n$ satisfies a central limit theorem and $\vert C_n\vert=\Theta(n)$ holds with probability $1-o(1)$. The precise statement is the following.
\begin{theorem}[\!\!{\cite[Eq.~(3.9) and~(3.10)]{Ben2}, \cite[Eq.~(5.6)]{PS}}]\label{theo:ctreesize0}
The size of the $C$-tree $\vert C_n\vert$ in a random \polya tree $T_n$ of size $n$ satisfies a central limit theorem where the expected value $\mathbb{E}\vert C_n\vert$ and the variance $\mathbb{V}\mbox{ar}\,\vert C_n\vert$ are asymptotically
\begin{align*}
\mathbb{E}\vert C_n\vert=\frac{2n}{b^2\rho}(1 + \CMcal{O}(n^{-1})),\quad\mbox{ and }\quad
\mathbb{V}\mbox{ar}\,\vert C_n\vert=\frac{11n}{12b^2\rho}
(1+\CMcal{O}(n^{-1})).
\end{align*}
Furthermore, for any $s$ such that $0<s<1/2$, with probability $1-o(1)$ we have
\begin{equation}\label{E:concen}
(1-n^{-s})\frac{2n}{b^2\rho}\le \vert C_n\vert \le (1+n^{-s})\frac{2n}{b^2\rho}.
\end{equation}
\end{theorem}
Random \polya trees belong to the class of random $\CMcal{R}$-enriched trees and we refer the readers to \cite{Ben2} for the proof of Theorem~\ref{theo:ctreesize0} in the general setting. Here we provide a proof of Theorem~\ref{theo:ctreesize0} to show the connection between a bivariate generating function and the normal distribution and to emphasize the simplifications for the concrete values of the expected value and variance in this case.

{\em Proof of Theorem~\ref{theo:ctreesize0} (see also \cite{Ben2} for a probabilistic proof)}. It follows from \cite[Th.~2.23]{Drmotabook} that the random variable $\vert C_n\vert$ satisfies a central limit theorem. In the present case, we set
$F(z,y,u)=zu\exp(y)D(z)$. It is easy to verify that $F(z,y,u)$ is an analytic function when $z$
and $y$ are near $0$ and that $F(0,y,u)\equiv 0$, $F(x,0,u)\not\equiv 0$ and all coefficients
$[z^ny^m]F(z,y,1)$ are real and nonnegative. From \cite[Th.~2.23]{Drmotabook} we know that
$T_c(z,u)$ is the unique solution of the functional identity $y=F(z,y,u)$. Since all coefficients
of $F_y(z,y,1)$ are nonnegative and the coefficients of $T(z)$ are positive as well as
monotonically increasing, this implies that $(\rho,T(\rho),1)$ is the unique solution of $F_y(z,y,1)=1$, which leads to the fact that $T(\rho)=1$. Moreover, the expected value is
\begin{align*}
\mathbb{E}\vert C_n\vert&=\frac{nF_u(z,y,u)}{\rho F_z(z,y,u)} \\
&=\frac{[z^n]\partial_u{T}_c(z,u)\vert_{u=1}}{[z^n]{T}(z)} \\
&=\left([z^n]\frac{{ T}(z)}{1-{T}(z)}\right)\left([z^n]{T}(z)\right)^{-1} \\
&=\frac{2n}{b^2\rho}(1 + \CMcal{O}(n^{-1})).
\end{align*}
The asymptotics are directly derived from~\eqref{eq:polyaasympt}. Likewise, we can compute the variance
\begin{align*}
\mathbb{V}\mbox{ar}\,\vert C_n\vert
&=\frac{[z^n]{T}(z)(1-{T}(z))^{-3}}{[z^n]{T}(z)}
-(\mathbb{E}\,\vert C_n\vert)^2
=\frac{11n}{12b^2\rho}
(1+\CMcal{O}(n^{-1})).
\end{align*}
Furthermore, $\vert C_n\vert$ is highly concentrated around $\mathbb{E}\,\vert C_n\vert$, which can be proved again by using Chebyshev's inequality. We set $\varepsilon_n=n^{-s}$ where $0<s<1/2$ and get
\begin{eqnarray*}
\mathbb{P}(\big\vert \vert C_n\vert-\mathbb{E}\vert C_n\vert\big\vert\ge \varepsilon_n\cdot\mathbb{E}\vert C_n\vert)
\le\frac{\mathbb{V}\mbox{ar}\vert C_n\vert}{\varepsilon_n^2\cdot(\mathbb{E}\vert C_n\vert)^2}
=\CMcal{O}(n^{2s-1})=o(1),
\end{eqnarray*}
which yields \eqref{E:concen}.
\qed

As a simple corollary, we also get the total size of all weighted $D$-forests in $T_n$. Let $\CMcal{D}_n$ denote the union of all $D$-forests in a random \polya tree $T_n$ of size $n$.
\begin{corollary}\label{C:1}
	The size of weighted $D$-forests in a random \polya tree of size $n$ satisfies a central limit theorem where the expected value $\mathbb{E}\vert\CMcal{D}_n\vert$ and the variance $\mathbb{V}\mbox{ar}\vert\CMcal{D}_n\vert$ are asymptotically
	\begin{align*}
	\mathbb{E}\vert\CMcal{D}_n\vert=n\left(1-\frac{2}{b^2\rho}\right)(1 + \CMcal{O}(n^{-1})),\quad\mbox{ and }\quad
	\mathbb{V}\mbox{ar}\vert\CMcal{D}_n\vert=\frac{11n}{12b^2\rho}
	(1+\CMcal{O}(n^{-1})).
	\end{align*}
\end{corollary}
Theorem~\ref{theo:ctreesize0} and Corollary~\ref{C:1} tell us that a random \polya tree $T_n$ consists mostly of a $C$-tree (proportion $\frac{2}{b^2\rho}$ comprising $\approx 82.2\%$ of the nodes) and to a small part of $D$-forests (proportion $1-\frac{2}{b^2\rho}$ comprising $ \approx 17.8\%$ of the nodes). Furthermore, the average size of a $D$-forest $F_n(v)$ attached to a random $C$-tree vertex in $T_n$ is $\frac{b^2\rho}{2}-1 \approx 0.216$, which indicates that on average the $D$-forest $F_n(v)$ is very small, although the maximal size of all $D$-forests in a random \polya tree $T_n$ reaches $\Theta(\log n)$.
\begin{remark}
Let us describe the connection of (\ref{eq:polyadeco}) to the Boltzmann sampler from \cite{PS}. We know that
$F(z,y,1)=z\Phi(y)D(z)$ where $\Phi(x)=\exp(x)$ and $y=T(z)$. By dividing both sides of this equation by $y=T(z)$, one obtains from~\eqref{eq:polyadecoA} that
\begin{align*}
1=\frac{z{D}(z)}{{T}(z)}\exp({T}(z))
=\exp(-{T}(z))\sum_{k\ge 0}\frac{{T}^k(z)}{k!},
\end{align*}
which implies that in the Boltzmann sampler $\Gamma T(x)$, the number of offspring contained in the $C$-tree $C_n$ is Poisson distributed with parameter $T(x)$. As an immediate result, the random $C$-tree contained in the \polya tree generated by the Boltzmann sampler $\Gamma T(\rho)$ corresponds to a critical Galton-Watson tree since the expected number of offspring is $F_y(z,y,1)=1$ for $(z,y)=(\rho,1)$.
\end{remark}
\section{\texorpdfstring{$D$-forests and $C$-trees}{D-forests and C-trees}}
\label{S:size-D}


In order to get a better understanding of $D$-forests and $C$-trees, we need to return to the original proof of \polya on the number of \polya trees \cite{poly37}. The important step is the treatment of tree automorphisms by the cycle index. 


As before, we denote by $\sigma_i$ the number of cycles of length $i$ of a permutation $\sigma$. Let $S_k$ be the symmetric group of order $k$. The {\em type} of a permutation $\sigma\in S_k$ is the $k$-tuple $(\sigma_1,\sigma_2,\ldots,\sigma_k)$. Note that $k=\sum_{i=1}^k i\sigma_i$.
\begin{definition}[Cycle index]\label{D:cin}
Let $G$ be a subgroup of the symmetric group $S_k$. Then the {\em cycle index} is
\begin{align*}
Z(G; s_1,s_2,\ldots,s_k)=\frac{1}{|G|}\sum_{\sigma\in G}s_1^{\sigma_1}s_2^{\sigma_2}\cdots s_k^{\sigma_k}.
\end{align*}
\end{definition}
Now we are ready to prove Theorem~\ref{T:2}.

\subsection{Proof of Theorem~\ref{T:2}} By P\'{o}lya's enumeration theory \cite{poly37}, 
the generating function $T(z)$ satisfies the functional equation
\begin{align*}\nonumber
T(z)&=z \sum_{k \ge 0} Z(S_k; T(z),T(z^2),\ldots,T(z^k))\\
&=z\sum_{k\ge 0}\frac{1}{k!}\sum_{\sigma\in S_k}(T(z))^{\sigma_1}(T(z^2))^{\sigma_2}\cdots (T(z^k))^{\sigma_k},
\end{align*}
which can be simplified to \eqref{E:penum1}, the starting point of our research, by a simple
calculation. However, this shows that the generating function of $D$-forests from~\eqref{eq:polyadecoA} is given by
\begin{align}
	D(z) &= \exp\left(\sum_{i=2}^\infty \frac{T(z^i)}{i}\right) \nonumber\\ 
				&= \sum_{k \geq 0} Z(S_k;0,T(z^2),\ldots,T(z^k)) 
			= \sum_{k \geq 0}\frac{1}{k!}\sum_{
			\sigma\in S_k \ \text{\scriptsize such that } \sigma_1=0}(T(z^2))^{\sigma_2}\cdots (T(z^k))^{\sigma_k}.  \label{Dforestgf}
\end{align}
The weight of a $D$-forest of size $n$ comprising $k$ trees is given by the ratio of fixed point free automorphisms over the total number of automorphisms.
This quotient equals the number of fixed point free permutations $\sigma \in S_k$ of the trees
which the forest consists of divided by the total number of orderings, $k!$, since the automorphisms of the subtrees of the root contribute to both the number of all and the number of fixed point free automorphisms of the forest. 
But this last quotient is precisely the coefficient of $z^n$ in the $k$th summand of \eqref{Dforestgf}. Thus 
\begin{align*}
	d_n = [z^n]D(z)=\sum_{\substack{F \in \operatorname{MSET}^{(\geq 2)}( \Tc )\\|F| = n}} \frac{|\{ \sigma \in \Aut(F) ~|~ \sigma_1 = 0\}|}{|\Aut(F)|}.
\end{align*} 
This proves the first assertion of Theorem~\ref{T:2}.

\begin{example}
The smallest $D$-forest is of size $2$, and it consists of a pair of single nodes, see Figure~\ref{fig:smalldforests}. This forest has only  
one fixed point free automorphism, thus $d_2 = 1/2$. For $n=3$ the forest
consists of three single nodes. The fixed point free permutations are the $3$-cycles, thus $d_3 =
2/6 = 1/3$. The case $n=4$ is more interesting. A forest consists either of four single nodes, or
of two identical trees, each consisting of two nodes and one edge. In the first case we have six
$4$-cycles  and three pairs of transpositions. In the second case we have one transposition swapping the two trees. Thus, $d_4 = \frac{6+3}{24} + \frac{1}{2} = \frac{7}{8}$. \hfill{$\diamondsuit$}
\end{example}


These results also yield a natural interpretation of $C$-trees. We recall that by definition
\begin{align*}
	T_c(z,u) &= \sum_{n \geq 0} t_{c,n}(u) z^n,
\end{align*}
where $t_{c,n}(u)=\sum_{k}c_{n,k}u^k$ is the polynomial marking the $C$-trees in \polya trees of size $n$. From the decompositions~\eqref{eq:polyadeco} and (\ref{E:bi1}) we get the first few terms:
\begin{align*}
	t_{c,1}(u) &= u, &
	t_{c,2}(u) &= u^2, &
	t_{c,3}(u) &= \frac{3}{2} u^3 + \frac{1}{2} u, &
	t_{c,4}(u) &= \frac{8}{3} u^4 + u^2 + \frac{1}{3} u %
	.
\end{align*}
Evaluating these polynomials at $u=1$ obviously returns $t_{c,n}(1) = t_n$, which is the number of \polya trees of size $n$.
Their coefficients, however, are weighted sums depending on the number of $C$-tree nodes. For a given \polya tree there are in general several ways to decide what is a $C$-tree node and what is a $D$-forest node. The possible choices are encoded in the automorphisms of the tree, and these are responsible for the above weights as well.

Let $T$ be a \polya tree and $\Aut(T)$ its automorphism group. For an automorphism $\sigma \in \Aut(T)$ the nodes which are fixed points of $\sigma$ are $C$-tree nodes. All other nodes are part of $D$-forests. Summing over all automorphisms and normalizing by the total number gives the $C$-tree generating polynomial for $T$:
\begin{align}
  \label{E:tTu}
	t_{T}(u) &= Z(\Aut(T); u, 1,\ldots,1) = \frac{1}{|\Aut(T)|} \sum_{\sigma \in \Aut(T)} u^{\sigma_1}.
\end{align}
The polynomial of $C$-trees in \polya trees of size $n$ is then given by
\begin{align*}
	t_{c,n}(u) &= \sum_{
	T \in \CMcal{T},\; \vert T\vert=n} t_{T}(u), 
\end{align*}
which completes the proof of the second assertion of Theorem~\ref{T:2}.

\begin{example}
	\label{ex:probinterpretation}
	For $n=3$ we have two \polya trees, namely the chain $T_1$, a chain made of three
vertices, and the cherry $T_2$, which is a root with two single vertices attached to it as
subtrees. Thus, $\Aut(T_1) = \{ \text{id} \}$, and $\Aut(T_2) = \{\text{id}, \sigma\}$, where $\sigma$ swaps the two leaves but the root is unchanged. Thus,
	\begin{align*}
		t_{T_1}(u) &= u^3, &
		t_{T_2}(u) & = \frac{1}{2} (u^3 + u).
	\end{align*}
	For $n=4$ we have the four \polya trees shown in Figure~\ref{F:2}. Their automorphism groups are given by $\Aut(T_1) = \Aut(T_2) = \{ \text{id} \}$, $\Aut(T_3) = \{ \text{id}, (v_3\, v_4) \} \cong S_2$, and
$$\Aut(T_4) = \{\text{id}, (v_2\,v_3), (v_3\,v_4), (v_2\,v_4), (v_2\, v_3\, v_4), (v_2\,v_4\,v_3) \} \cong S_3.$$
This gives
	\begin{align*}
		t_{T_1}(u) &= t_{T_2}(u) = u^4, &
		t_{T_3}(u) &= \frac{1}{2} (u^4+u^2), &
		t_{T_4}(u) &= \frac{1}{6} (u^4 + 3u^2 + 2u).
	\end{align*}
This enables us to give a probabilistic interpretation of the composition
scheme~\eqref{eq:polyadeco}. For a given tree the weight of $u^k$ is the probability that the
underlying $C$-tree is of size $k$. In particular, $T_1$ and $T_2$ do not have $D$-forests. The
tree $T_3$ consists of a $C$-tree with four or with two nodes, each case with probability $1/2$. In
the second case, as there is only one possibility for the $D$-forest, it consists of the pair of
single nodes which are the leaves. Finally, the tree $T_4$ has either four $C$-tree nodes with
probability $1/6$, two with probability $1/2$, or only one with probability $1/3$. These decompositions are shown in Figure~\ref{F:0}. \hfill{$\diamondsuit$}
\end{example}
\begin{figure}[htbp]
\begin{center}
\hspace{1.6cm}
\includegraphics[scale=1.2]{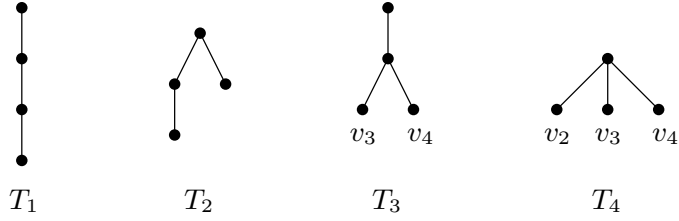}
\caption{All \polya trees of size $4$\label{F:2}.}
\end{center}
\end{figure}

In the same way as we got the composition scheme in \eqref{eq:polyadeco}, we can rewrite $T_c(z,u)$ from~\eqref{E:bi1} into $T_c(z,u)=C(uzD(z))$. 
The expected total weight of all $C$-trees contained in all \polya trees of size $n$ is the $n$-th coefficient of $T_c(z)$, which is
\begin{align}
	\label{E:expTc}
	T_c(z) := \left.\frac{\partial}{\partial u} T_c(z,u) \right|_{u=1} = \frac{T(z)}{1 - T(z)} = z + 2z^2 + 5z^3 + 13z^4 + 35 z^5 + 95 z^6 + \cdots.
\end{align}
Let us explain why these numbers are integers, although the coefficients of $t_{c,n}(u)$ are in general not. We will show an even stronger result. 

\pagebreak[2]

\begin{lemma}
	\label{L:markingTc}
	Let $T$ be a tree and $\CMcal{P}(T)$ be the set of all trees with one single pointed (or colored) node which can be generated from $T$. Then 
	for all $T \in \CMcal{T}$ we have $t_T'(1) = |\CMcal{P}(T)|$.
\end{lemma}
\begin{proof}
From~\eqref{E:tTu} we get 
$
	t_T'(1) = \sum_{\sigma \in \Aut(T)} \frac{\sigma_1}{\vert\Aut (T)\vert}
$
is the expected number of fixed points in a uniformly at random chosen automorphism of $T$. The associated random variable $C_T$ is defined in~\eqref{eq:probCT}.
We will prove $\E(C_T)=|\CMcal{P}(T)|$ by induction on the size of $T$.

The most important observation is that only if the root of a subtree is a fixed point, its children can also be fixed points. Obviously, the root of the tree is always a fixed point.

For $|T| = 1$, the claim holds as $\E(C_T)=1$ and there is just one tree with a single node and a marker on it. For larger $T$ consider the construction of \polya trees. A \polya tree consists of a root $T_0$ and its children, which are a multiset of smaller trees. Thus, the set of children is of the form
\begin{align*}
	\{T_{1,1},\ldots,T_{1,k_1},T_{2,1},\ldots,T_{2,k_2},\ldots,T_{r,1},\ldots,T_{r,k_r} \}, \quad \text{ with } \quad T_{i,j} \in \CMcal{T},
\end{align*}
and where trees with the same first index are isomorphic.
On the level of children, the possible behaviors of automorphisms are permutations within the same class of trees. In other words, an automorphism may interchange the trees $T_{1,1},\ldots,T_{1,k_1}$ in $k_1!$ many ways, etc. Here the main observation comes into play: only subtrees of which the root is a fixed point might also have other fixed points. Thus, the expected number of fixed points is given by the expected number of fixed points in a random permutation of $S_{k_i}$ times the expected number of fixed points in $T_{k_i}$. By linearity of expectation we get
\begin{align*}
	\E(C_T) = \E(C_{T_0})+\sum_{i = 0}^r \underbrace{\E(\text{number of fixed points in }S_{k_i})}_{=1} \E(C_{T_i}),
\end{align*}
where $\E(C_{T_i})=\E(C_{T_{i,j}})$ for all $1\le j\le k_i$ and $\E(C_{T_0})=1$ because the root is a fixed point of any automorphism. Since the expected number of fixed points for each permutation is $1$, we get on average $1$ representative for each class of trees. This is exactly the operation of labeling one tree among each equivalence class.
Finally, by induction the claim holds.
%
%
\end{proof}
This completes the proof of Theorem~\ref{T:2}.
\qed

\medskip
As an immediate consequence of Lemma~\ref{L:markingTc}, $t_{c,n}'(1)$ counts the number of \polya trees with $n$ nodes and a single labeled node (see OEIS~A$000107$, \cite{Sloane}).
This also explains the construction of non-empty sequences of trees in~\eqref{E:expTc}: Following the connection of \cite[pp. 61--62]{Bergeronbook} one can draw a path from the root to each labeled node. The nodes on that path are the roots of a sequence of \polya trees.

\begin{remark}
	Note that Lemma~\ref{L:markingTc} also implies that the total number of fixed points in all automorphisms of a tree is a multiple of the number of automorphisms.
\end{remark}
\begin{remark}
Lemma~\ref{L:markingTc} can also be proved by considering cycle-pointed \polya trees; see \cite[Section~3.2]{Kang} for a full description. Let $(T,c)$ be a cycle-pointed structure considered up to symmetry where $T$ is a \polya tree and $c$ is a cycle of an automorphism $\sigma\in \Aut(T)$. Then, the number of such cycle-pointed structures $(T,c)$ where $c$ has length $1$ is exactly the number $t_T'(1)$.

\end{remark}
Let us analyze the $D$-forests in $T_n$ more carefully. We want to count the number of $D$-forests
that have size $m$ in a random \polya tree $T_n$. Therefore, we label such $D$-forests with an additional parameter $u$ in \eqref{eq:polyadeco}. From the bivariate generating function (\ref{E:bi2}) we can recover the probability $\mathbb{P}[\vert F_n(v)\vert=m]$ to generate a $D$-forest of size $m$ in the Boltzmann sampler from \cite{PS}.

\subsection{Proof of Theorem~\ref{T:312}}
	The first result is a direct consequence of \eqref{eq:polyadeco}, where only vertices with weighted $D$-forests of size $m$ are marked. For the second result we differentiate both sides of (\ref{E:bi2}) and get
	\begin{align*}
		T^{[m]}_u(z,1) &= \frac{T(z)}{1-T(z)} \frac{d_m z^m}{D(z)} = T_c(z) \frac{d_m z^m}{D(z)}.
	\end{align*}
	Then, the sought probability is given by
	\begin{align*}
		\PR\left[ \vert F_n(v)\vert = m \right] = \frac{ [z^n] T^{[m]}_u(z,1)}{ [z^n] T_c(z)} =  \frac{d_m \rho^m}{D(\rho)} \left( 1 + \LandauO\left(n^{-1}\right)\right).
	\end{align*}
	For the last equality we used the fact that $D(z)$ is analytic in a neighborhood of $z=\rho$.

Let $P_n(u)$ be the probability generating function for the size of a weighted $D$-forest $F_n(v)$ attached to a vertex $v$ of $C_n$ in a random \polya tree $T_n$. From the previous theorem it follows that
\begin{align*}
	P_n(u) &= \sum_{m \geq 0} \frac{ [z^n] T^{[m]}_u(z,1)}{ [z^n] T_c(z)} u^m
	        = \frac{ [z^n] T_c(z) \frac{D(zu)}{D(z)}}{ [z^n] T_c(z)} = \frac{D(\rho u)}{D(\rho)} \left( 1 + \LandauO\left(n^{-1}\right)\right).
\end{align*}
This is exactly \cite[Eq.~(5.2)]{PS}.
\qed

\medskip
Summarizing, we state in Table~\ref{tab:decoratedcnodes} the asymptotic probabilities that a weighted $D$-forest $F_n(v)$ in $T_n$ has size equal to or greater than $m$.
\begin{table}[htbp]
	\begin{center}
	\begin{tabular}{|c||c|c|c|c|c|c|c|c|c|}
        \hline $m$ & $0$ & $1$ & $2$ & $3$ & $4$ & $5$ & $6$ & $7$ \\
        \hline $\PR[ \vert F_n(v)\vert = m] \approx$ & $0.9197$ & $0.0000$ & $0.0526$ & $0.0119$ & $0.0105$ & $0.0015$ & $0.0027$ & $0.0003$\\
        \hline $\PR[ \vert F_n(v)\vert \ge m] \approx$ & $1.0000$ & $0.0803$ & $0.0803$ & $0.0277$ & $0.0161$ & $0.0060$ & $0.0041$ & $0.0014$\\
		\hline
	\end{tabular}
	\end{center}
	\caption{The probability that a weighted $D$-forest $F_n(v)$ has size equal to or greater than $m$ when $0\le m\le 7$.}
	\label{tab:decoratedcnodes}
\end{table}

As most of the vertices in $C_n$ have an empty $D$-forest, it is also interesting to condition on the non-empty ones only. Its generating function is given by
\begin{align*}
	Q_n(u) &= \sum_{n \geq 2} \PR[ \vert F_n(v) \vert = m~\big|~\vert F_n(v) \vert > 0] u^m = \frac{D(\rho u) -1}{D(\rho) -1}\left( 1 + \LandauO\left(n^{-1}\right)\right).
\end{align*}
Its first values are listed in Table~\ref{tab:decoratedcnodes2}. It is interesting to see in these tables that the sequence of probabilities is not decreasing in $m$. Additionally, we deduce that more than $80\%$ of the used $D$-forests are $D$-forests of size $2$ and $3$. These are two or three copies of a single node.
\begin{table}[htbp]
	\begin{center}
	\begin{tabular}{|c||c|c|c|c|c|c|c|c|c|}
        \hline $m$ & $2$ & $3$ & $4$ & $5$ & $6$ & $7$ & $8$ & $9$ \\
        \hline $\PR[ \vert F_n(v) \vert = m~\big|~\vert F_n(v) \vert > 0] \approx$ & $0.656$ & $0.148$ & $0.131$ & $0.019$ & $0.034$ & $0.003$ & 0.007 & 0.001 \\
		\hline
	\end{tabular}
	\end{center}
	\caption{The probability that a weighted $D$-forest $F_n(v)$ has size equal to or greater than $m$ when $0\le m\le 7$.}
	\label{tab:decoratedcnodes2}
\end{table}

\subsection{Properties of $D$-forests}
Let us start with a short analysis of $D(z)$.
\begin{lemma}
	\label{lem:D(z)}
	The generating function $D(z)$ of $D$-forests has radius of convergence $\sqrt{\rho}$. It has two dominant singularities at $z=\pm \sqrt{\rho}$. Let $\xi(z) = e^{\frac{T(z^3)}{3} + \frac{T(z^4)}{4} + \cdots}$, which is analytic for $|z|<\rho^{1/3}$. Then,
	\begin{align}
		\label{eq:dnasympt}
		d_n =  \left(\xi(\sqrt{\rho}) + (-1)^n \xi(-\sqrt{\rho})\right) b\sqrt{\frac{\rho e}{8\pi}} \frac{\rho^{-n/2}}{\sqrt{n^3}} \left(1 + \LandauO\left(\frac{1}{n}\right)\right).
	\end{align}
	Furthermore, we have $D(\rho) = \frac{1}{e\rho}$ and $D'(\rho) = \frac{1}{e\rho^2}\left(\frac{b^2\rho}{2}-1\right)$.
\end{lemma}

\begin{proof}
	The key essence to this result is the elementarily checked fact that if $T(z)$ has radius of convergence $\rho$, then $T(z^2)$ will have radius of convergence $\sqrt{\rho}$. Therefore, $\pm \sqrt{\rho}$ are the dominant singularities, as $T(z)$ has a unique singularity at $z=\rho$.
	
	The asymptotic expansions are then derived from \eqref{eq:polyaasympt} as
	\begin{align*}
		T(z^2) &= 1 - b\sqrt{2 \rho} \left(1 \mp \frac{z}{\sqrt{\rho}}\right)^{1/2} + \LandauO\left(1 \mp \frac{z}{\sqrt{\rho}}\right), & \text{ for } z \to \pm \sqrt{\rho}.		
	\end{align*}
	Next, note that $\xi(z)$ is analytic for $|z|<\rho^{1/3}$ due to the same reasoning as above. Thus, the asymptotic expansion of $D(z) = e^{\frac{T(z^2)}{2}}\xi(z)$ is derived by combining the contributions on the two dominant singularities. 	
	
	Finally, the values for $D(\rho)$ and $D'(\rho)$ are derived from \eqref{eq:polyadeco}.
\end{proof}

We want to determine the number of trees in a given $D$-forest. Let $d_{n,k}$ be the weight of $D$-forests of size $n$ consisting of $k$ trees. Then, by~\eqref{eq:polyadecoA} the bivariate generating function satisfies
\begin{align*}
	D(z,v) &= \sum_{n,k \geq 0} d_{n,k} z^n v^k = \exp\left(\sum_{i=2}^{\infty} v^i \frac{T(z^i)}{i}\right).
\end{align*}

\begin{theorem}
Let $X_{n}$ be the random variable for the number of trees in a $D$-forest of size $n$, {\em i.e.},~$\PR[X_n = k] := \frac{d_{n,k}}{d_n}$. Then we have
\begin{align*}
	\E X_n &=
		\begin{cases}
			3 + \mu_0 + \LandauO\left(n^{-1}\right) \approx 3.2715 + \LandauO\left(n^{-1}\right), & \text{ for $n$ even},\\
			3 + \mu_1 + \LandauO\left(n^{-1}\right) \approx 6.7852 + \LandauO\left(n^{-1}\right), & \text{ for $n$ odd},\\
		\end{cases}
\end{align*}
with
\begin{align*}
	\mu_0 &:= \frac{\xi(\sqrt{\rho}) \gamma(\sqrt{\rho}) + \xi(-\sqrt{\rho}) \gamma(-\sqrt{\rho})}{\xi(\sqrt{\rho}) + \xi(-\sqrt{\rho}) }, &
	\mu_1 &:= \frac{\xi(\sqrt{\rho}) \gamma(\sqrt{\rho}) - \xi(-\sqrt{\rho}) \gamma(-\sqrt{\rho})}{\xi(\sqrt{\rho}) - \xi(-\sqrt{\rho}) }, \\
	\xi(z) &= \exp\Big( \sum_{i \geq 3} \frac{T(z^i)}{i} \Big), &
	\gamma(z) &= \sum_{i \geq 2} T(z^i).
\end{align*}
\end{theorem}

\begin{proof}
	The correspondence with generating functions gives
	\begin{align*}
		\E X_n &= \frac{[z^n] D_v(z,1)}{[z^n] D(z)}
		        = \frac{[z^n] D(z) \sum_{i \geq 2} T(z^i)}{[z^n] D(z)}.
	\end{align*}
	As $T(z)$ has a unique singularity at $\rho$, $T(z^k)$ is singular at $\omega^k
\rho^{1/k}$ where $\omega = \exp(2\pi i/k)$ is a $k$-th root of unity. By linearity of the
coefficient extraction operator all that remains is to consider $D(z)T(z^2)$. By \eqref{eq:dnasympt} we get
	\begin{align*}
		D(z) T(z^2) &= -3b\sqrt{\frac{e \rho}{2}} \left(1 \mp \frac{z}{\sqrt{\rho}}\right)^{1/2} + \LandauO\left(1 \mp \frac{z}{\sqrt{\rho}}\right), & \text{ for } z \to \pm \sqrt{\rho}.	
	\end{align*}
	Therefore, by standard techniques of singularity analysis~\cite{FS} we get
	\begin{align*}
		\frac{[z^n] D(z) T(z^2)}{[z^n]D(z)} &= 3 + \LandauO\left(n^{-1}\right).
	\end{align*}
	The fluctuating constant $\mu_n$ arises from the second part $\frac{[z^n] D(z)
\gamma(z)}{[z^n] D(z)}$. Due to the reasoning above $\gamma(z)$ is analytic for $|z|<\rho^{1/3}$.
Thus, we can again use \eqref{eq:dnasympt} to combine the singular expansions of $D(z)$ at $\pm \sqrt{\rho}$ with the analytic expansion of $\gamma(z)$.
	
	For the computations of the approximate values we used Maple.
\end{proof}

As a next step we investigate the number of trees of a random $D$-forest in a random \polya tree of size $n$. Let $t_{n,k}$ be the weight of \polya trees of size $n$ with having $k$ trees in their $D$-forests. Then, by~\eqref{eq:polyadeco} the bivariate generating function satisfies
\begin{align*}
	T(z,v) &= \sum_{n,k \geq 0} t_{n,k} z^n v^k = C(z D(z,v)) = z D(z,v) e^{T(z,v)}.
\end{align*}

\begin{theorem}\label{theo:dtreenumber}
The total number $Y_n$ of trees of all $D$-forests in a random \polya tree $T_n$ of size $n$ satisfies a central limit theorem where the expected value $\mathbb{E} Y_n$ and the variance $\mathbb{V}\mbox{ar}\,Y_n$ are asymptotically
\begin{align*}
\mathbb{E}Y_n=\frac{2\gamma(\rho)}{b^2 \rho} n (1 + \CMcal{O}(n^{-1})),\quad\mbox{ and }\quad
\mathbb{V}\mbox{ar}\,Y_n=\sigma^2 n
(1+\CMcal{O}(n^{-1})),
\end{align*}
with $\sigma^2 = \frac{2}{b^2\rho} \left(\frac{(2b^3 \rho + 72 d \rho + 18 b) \gamma(\rho)^2}{9 b^3 \rho} - \frac{4\gamma'(\rho)\gamma(\rho)}{b^2} + \gamma_2(\rho)\right) \approx 0.26718$, and $\gamma_2(z) = \sum_{i \geq 2} i T(z^i)$.
Furthermore, for any $s$ such that $0<s<1/2$, with probability $1-o(1)$ we have
\begin{eqnarray}\label{dtreenumber:concen}
(1-n^{-s})\frac{2\gamma(\rho)}{b^2\rho}n \le Y_n \le (1+n^{-s})\frac{2\gamma(\rho)}{b^2\rho}n.
\end{eqnarray}
\end{theorem}

\begin{proof}
	The proof uses the same techniques as the one of Theorem~\ref{theo:ctreesize0}. In
particular, it follows again from~\cite[Th.~2.23]{Drmotabook} that $Y_n$ satisfies a central limit
theorem. Here, we set $F(z,y,v) = z e^y D(z,v)$. The technical conditions are easy to check, and
we know that $T(z,v)$ is the unique solution of $y = F(z,y,v)$.
\end{proof}

Theorems~\ref{theo:ctreesize0} and \ref{theo:dtreenumber} tell us that in a \polya tree of size
$n$ there are on average $\frac{2\gamma(\rho)}{b^2\rho}n \approx 0.15776n$ trees in $D$-forests,
and $\gamma(\rho) \approx 0.191837$ trees in the $D$-forest of a $C$-tree node. Comparing this
number with the average size of the $D$-forest $\frac{b^2\rho}{2}-1 \approx 0.216$ of a $C$-tree
node, we conclude that most trees consist of only one node. In particular, as every tree consists
of at least a root node, a component of a $D$-forest has on average approximately $0.024167$
non-root nodes. This implies, that on average only every $42^\text{nd}$ $D$-tree has more than one
node. (This is because $(0.024167)^{-1}\approx 42$.)

\section{Other P\'{o}lya structures} \label{other}
\subsection{P\'{o}lya trees with outdegree restriction}
Our work can be extended to $\Omega$-\polya trees in the same way, so we omit the proof.
For any $\Omega\subseteq \mathbb{N}_0=\{0,1,\ldots\}$ such that $0\in\Omega$ and $\{0,1\}\ne \Omega$, {\em an $\Omega$-\polya tree} is a rooted unlabeled tree considered up to symmetry and with outdegree set $\Omega$. When $\Omega=\mathbb{N}_0$, a $\mathbb{N}_0$-\polya tree is a \polya tree.

Let $a_{n}$ be the number of \polya trees of size $n$ with outdegree set $\Omega$, and $\CMcal{A}(z)$ the associated generating function. That is, $a_{n}=[z^n]\CMcal{A}(z)$.
From \polya enumeration theory \cite{poly37} and Burnside's Lemma, the generating function $\CMcal{A}(z)$ satisfies the functional equation
\begin{align}\nonumber
\CMcal{A}(z)&=z\cdot \sum_{k\in \Omega} Z(S_k;\CMcal{A}(z),\CMcal{A}(z^2),\ldots,\CMcal{A}(z^k))\\
\label{E:penum-ome}&=z\sum_{k\in \Omega}\frac{1}{k!}\sum_{\sigma\in S_k}(\CMcal{A}(z))^{\sigma_1}(\CMcal{A}(z^2))^{\sigma_2}\cdots (\CMcal{A}(z^k))^{\sigma_k}.
\end{align}
Proposition~\ref{P:uni} has been also used in \cite{PS}. It was actually implicitly stated in \cite{Bell:06} and fits into the general theorem on implicit functions in \cite{Drmotabook,FS}.
\begin{proposition}\label{P:uni}
Let $\tau$ be the unique dominant singularity of $\CMcal{A}(z)$. Then
$0<\tau<1$ and $0<\CMcal{A}(\tau)<\infty$. Furthermore, $\tau$ is the unique real solution of
\begin{align}\label{E:sol1ome}
\sum_{k\in\Omega}\frac{\partial}{\partial x}Z(S_k;x,\CMcal{A}(\tau^2),\ldots,\CMcal{A}(\tau^k))
\big|_{x=\CMcal{A}(\tau)}=\frac{1}{\tau}.
\end{align}
and $\CMcal{A}(z)$ has a local expansion of the form
\begin{align}\label{eq:polyaasymptome}
\CMcal{A}(z) &= \CMcal{A}(\tau)  - b_1\left(\tau-z\right)^{1/2} + c_1(\tau-z) + \LandauO\left((\tau-z)^{3/2}\right)
\end{align}
where $b_1>0$ is a constant and $a_{n}=[z^n]\CMcal{A}(z)$ is asymptotically
\begin{align}\label{E:ao}
a_{n}=\frac{b_{1}\sqrt{\tau}}{2\sqrt{\pi}}
n^{-3/2}\tau^{-n}(1+\CMcal{O}(n^{-1})).
\end{align}
\end{proposition}
Similar to $T_n$, we consider a random P\'{o}lya tree of size $n$ with outdegree set $\Omega$, denoted by $T_n^{(d)}$, which is a tree that is selected uniformly at random from all P\'{o}lya trees of $n$ vertices and with outdegree set $\Omega$. Similar to $L_n$, let $L_n^{(d)}$ be the maximal size of a $D$-forest contained in $T_n^{(d)}$. In the same way as Theorem~\ref{T:1}, we have
\begin{theorem}\label{T:1ome}
For $0<s<1$,
\[
(1-(\log n)^{-s})\left(\frac{-2\log n}{\log\tau}\right)\le L_n^{(d)} \le (1+(\log n)^{-s})\left(\frac{-2\log n}{\log\tau}\right)
\]
holds with probability $1-o(1)$.
\end{theorem}
Note that $\tau$ is determined by (\ref{E:sol1ome}). From (\ref{E:penum-ome}) we see that every P\'{o}lya tree with outdegree restriction $\Omega$ is a multiset of small P\'{o}lya trees with outdegree restriction $\Omega$. We first consider the bivariate generating function
\begin{align}\label{E:bvg1-ome}
\CMcal{A}(z,u)=uz\cdot \sum_{k\in\Omega}
Z(S_k;\CMcal{A}(z,u),\CMcal{A}_{\Omega}(z^2),\ldots,\CMcal{A}_{\Omega}(z^k))
\end{align}
For a random P\'{o}lya tree $T_n^{(d)}$ with outdegree restriction $\Omega$, we then select one automorphism of $T_n^{(d)}$ uniformly at random and all fixed points of such an automorphism form a random $C$-tree, denoted by $C_n^{(d)}$. It was also shown in \cite{PS} that the size $\vert C_n^{(d)}\vert$ in $T_n^{(d)}$ satisfies a central limit theorem and $\vert C_n^{(d)}\vert=\Theta(n)$ holds with probability $1-o(1)$.
\begin{theorem}[{\cite[Eq.~(3.9) and~(3.10)]{Ben2}, \cite[Eq.~(5.6)]{PS}}]\label{theo:ctreesize0-ome}
The size of the $C$-tree $\vert C_n^{(d)}\vert$ in a random \polya tree $T_n^{(d)}$ of size $n$ satisfies a central limit theorem where the expected value $\mathbb{E}\vert C_n^{(d)}\vert$ and the variance $\mathbb{V}\mbox{ar}\,\vert C_n^{(d)}\vert$ are asymptotically
\begin{align*}
\mathbb{E}\,\vert C_n^{(d)}\vert=\frac{n}{1+\mu}(1 + \LandauO(n^{-1}))\,\mbox{ where }
\mu=\frac{\tau^2}{\CMcal{A}(\tau)}\sum_{k\in \Omega}
\frac{\partial}{\partial x}Z(S_k;\CMcal{A}(\tau),\CMcal{A}(x^2),\ldots,\CMcal{A}(x^k))
\big|_{x=\tau}
\end{align*}
and the variance is $\mathbb{V}\mbox{ar}\,\vert C_n^{(d)}\vert=\sigma^2 n$ where $\sigma>0$. Furthermore, for any $s$ such that $0<s<1/2$, with high probability we have
\[
(1-n^{-s})\frac{n}{1+\mu}\le \vert C_n^{(d)}\vert \le (1+n^{-s})\frac{n}{1+\mu}.
\]
\end{theorem}

\begin{example}
For $\Omega=\mathbb{N}_0-\{1\}$, any $\Omega$-\polya tree is a \polya tree without nodes of degree $1$, which is also called {\em a hierarchy}. Let $T^*(z)$ be the ordinary generating function of hierarchies. Then if we remove the root of a hierarchy, we are left with a multiset of smaller hierarchies and the number of such subtrees is at least $2$. That is,
\begin{align}\label{E:hier}
T^*(z)&=z\exp\left(\sum_{i=1}^{\infty}\frac{T^*(z^i)}{i}\right)-zT^*(z)\\
\nonumber&=\frac{z}{1+z}\exp\left(\sum_{i=1}^{\infty}\frac{T^*(z^i)}{i}\right)
=z+z^3+z^4+2z^5+3z^6+\cdots.
\end{align}
The generating function of hierarchies was also derived in \cite{Antoine} where the size of a hierarchy is defined as the number of leaves, instead of the number of nodes. From (\ref{E:sol1ome}) we find that $\tau$ is the unique solution of
\begin{align*}
&\quad \sum_{k\in\mathbb{N}_0}\frac{\partial}{\partial x}Z(S_k;x,T^*(\tau^2),\ldots,T^*(\tau^k))\vert_{x=T^*(\tau)}
-\frac{\partial}{\partial x}Z(S_1;x)\vert_{x=T^*(\tau)}\\
&=\exp\left(\sum_{i=1}^{\infty}\frac{T^*(z^i)}{i}\right)-T^*(\tau)
=\tau^{-1}T^*(\tau)=\tau^{-1}.
\end{align*}
This yields $T^*(\tau)=1$. If we again differentiate both sides of (\ref{E:hier}) and take the $n$-th coefficient of $z$, we get a recursion of hierarchies, namely, let $t^*(n)=[z^n]T^*(z)$. Then we have $t^*(1)=1$, $t^*(2)=0$ and for $n\ge 3$,
\begin{align*}
t^*(n)=\frac{1}{n-1}\sum_{i=1}^{n-2}(t^*(n-i)+t^*(n-i-1))\sum_{m\vert i} mt^*(m)+\frac{1}{n-1}\sum_{\substack{m\vert (n-1)\\ m\ne (n-1)}} mt^*(m).
\end{align*}
With the help of this recursion, we can use Maple to generate the numbers of hierarchies, by which we can find the approximate solution $\tau\approx 0.4580838$ of $T^*(\tau)=1$. Furthermore, for the case of hierarchies, we compute $\mu$ in Theorem~\ref{theo:ctreesize0-ome}, which is
\begin{align*}
\mu=\frac{\tau^2}{T^*(\tau)}\exp(T^*(\tau)) 
\left(\exp(\sum_{i=2}^{\infty}\frac{T^*(z^i)}{i})\right)'_{z=\tau}
\approx 0.6701252,
\end{align*}
where we used $T^*(\tau)=1$.
\hfill $\diamondsuit$
\end{example}

\begin{example}
For $\Omega=\{0,2\}$, any $\Omega$-\polya tree is a binary \polya tree. Let $T_2(z)$ be the ordinary generating function of binary \polya trees. Then we have
\begin{align}\label{E:bpo}
T_2(z)=z+\frac{1}{2}z (T_2(z))^2+\frac{1}{2}z T_2(z^2).
\end{align}
From (\ref{E:sol1ome}) we find that $\tau$ is the unique solution of
\begin{align}\label{E:unib}
\frac{\partial}{\partial x}(Z(S_{0})+Z(S_2;x,T_2(\tau^2)))\vert_{x=T_2(\tau)}=T_2(\tau)=\tau^{-1},
\end{align}
and as before we can derive a recursion from (\ref{E:bpo}), namely, let $t_2(n)=[z^n]T_2(z)$. Note that every binary \polya trees has an even number of nodes, that is, for even $n$, $t_2(n)=0$. For odd $n$, $n\ge 3$, we have
\begin{align*}
t_2(n)=\frac{1}{2}\sum_{i=1}^{n-2}t_2(i)t_2(n-1-i)
+\frac{1}{2}t_{\lfloor\frac{n-1}{2}\rfloor},
\end{align*}
and $t_2(1)=1$. With the help of this recursion, we can use Maple to generate the numbers of binary \polya trees, by which we can find the approximate solution $\tau\approx 0.6348553$ of $T_2(\tau)=\tau^{-1}$. Furthermore, for the case of binary \polya trees, we compute $\mu$ in Theorem~\ref{theo:ctreesize0-ome}, which is
\begin{align*}
\mu=\frac{\tau^2}{T_2(\tau)}(1+\frac{1}{2}\frac{\partial}{\partial x}T_2(x^2)\vert_{x=\tau})=\tau^3(1+\tau T_2'(\tau^2))\approx 0.5330644,
\end{align*}
where we used $T_2(\tau)=\tau^{-1}$ from (\ref{E:unib}).
\hfill $\diamondsuit$
\end{example}

%

\subsection{Rooted identity trees}
Rooted identity trees are a further \polya structure which has been listed and treated in \cite{Antoine}. They do not fit into the framework of $\Omega$-\polya trees. Nevertheless, there are some analogies to \polya trees when our framework is applied. This section presents a discussion of what happens when we use our framework on rooted identity trees, which will eventually lead to a combinatorial interpretation of OEIS sequence A052806. 

A rooted identity tree is a P\'{o}lya tree whose automorphism group is the identity group. Let $R(z)$ denote the ordinary generating function of rooted identity trees. Then we can identify every rooted identity tree as a powerset of smaller rooted identity trees, which is a multiset of rooted identity trees that involves no repetition; see \cite{FS}. This gives
\begin{align}\label{E:rid}
R(z)&=z\exp\left(\sum_{i=1}^{\infty}(-1)^{i-1}\frac{R(z^i)}{i}\right)
=z\exp(R(z))\exp\left(\sum_{i=2}^{\infty}(-1)^{i-1}\frac{R(z^i)}{i}\right).
\end{align}
The first few terms of $R(z)$ are
\begin{align*}
R(z)=z+z^2+z^3+2z^4+3z^5+6z^6+12z^7+25z^8+52z^9+\cdots.
\end{align*}
Note that the sets of rooted identity trees can be generated as multiset of rooted identity trees
with signed weights, realizing an inclusion--exclusion process. Let $\CMcal{R}$ be the set of all rooted identity trees. We consider the multiset of the elements in $\CMcal{R}$, which is denoted by $\mbox{MSET}(\CMcal{R})$, and as before we use $\mbox{MSET}^{(\ge 2)}(\CMcal{R})$ to denote the multiset of rooted identity trees where each tree appears at least twice if it appears at all. Here a $D^*$-forest of size $n$ is an element of $\mbox{MSET}^{(\ge 2)}(\CMcal{R})$. The generating function for the $D^*$-forests of rooted identity trees is
\begin{align*}
D^*(z)=\exp\left(\sum_{i=2}^{\infty}(-1)^{i-1}\frac{R(z^i)}{i}\right)
&=\sum_{k\ge 0}Z(S_k;0,-R(z^2),\cdots,(-1)^{k-1}R(z^k))\\
&=\sum_{k\ge 0}\frac{1}{k!}\sum_{\substack{\sigma\in S_k\\ \sigma_1=0}}
(-1)^{\sigma_2+\sigma_4+\cdots}(R(z^2))^{\sigma_2}\cdots (R(z^k))^{\sigma_k}.
\end{align*}
Then their cumulative weights are given by
\begin{align*}
	d^*_n=[z^n]D^*(z)= \sum_{\substack{F \in \operatorname{MSET}^{(\geq 2)}( \CMcal{R} )\\|F| = n}} \frac{1}{|\Aut(F)|}\sum_{\sigma\in \Aut(F)\text{\scriptsize such that } \sigma_1=0}(-1)^{\sigma_2+\sigma_4+\cdots}
\end{align*}
and a single term of this sum is the (signed) weight of a $D^*$-forest $F$. The first few terms of $D^*(z)$ are
\begin{align*}
D^*(z)=1-\frac{1}{2}z^2+\frac{1}{3}z^3-\frac{5}{8}z^4+\frac{1}{30}z^5
+\frac{11}{144}z^6-\frac{139}{840}z^7+\cdots.
\end{align*}
\begin{example}
The smallest $D^*$-forest is of size $2$, and it consists of a pair of single nodes. The only 
fixed point free automorphism is a transposition, thus $d^*_2=-1/2$. For $n=3$, the $D^*$ consists
of three single nodes. The only fixed point free automorphisms are the $3$-cycles, thus
$d^*_3=2/6=1/3$. 

For $n=4$, a $D^*$-forest consists either of four single nodes, or of two identical trees, each
consisting of two nodes and one edge. In the first case we have six $4$-cycles  and three pairs of
transpositions. In the second case we have one transposition swapping the two trees. Thus, $d^*_4 = (-6+3)/24-1/2 =-5/8$. \hfill{$\diamondsuit$}
\end{example}
Now we define bivariate generating function in analogy to what we did for \polya trees. Define a function via the functional equation
\begin{align*}
R_c(z,u)=zu\exp(R_c(z,u))\exp\left(\sum_{i=2}^{\infty}(-1)^{i-1}
\frac{R(z^i)}{i}\right).
\end{align*}
and set
\begin{align}
R_c(z,u)=\sum_{n\ge 0}r_{c,n}(u)z^n\quad \mbox{ where }\quad
r_{c,n}(u)=\sum_{\substack{T\in\text{\scriptsize MSET}(\CMcal{R})\\ \vert T\vert=n}}r_T(u). \label{Rc}
\end{align}
If we set $u=1$ then we get back the generating function of rooted identity trees. Note that the coefficients $[u^k]r_{c,n}(u)$ do not have the nice interpretation as cumulative weight of all $C$-trees identity trees of size $k$ contained in rooted identity trees of size $n$. This is because a rooted identity tree has only the trivial automorphism which means that every vertex is a fixed point and thus the whole tree is its $C$-tree. But this is in contradiction to $R(z)\neq C(z)$. 

On the other hand, we have $R(z)=C(zD^*(z))$ meaning that a rooted identity tree is a $C$-tree to which $D^*$-forests have been attached. But due to the signed weights the cumulative weight of all decompositions of \polya tree $T$ into a $C$-tree and a set of $D^*$-forests is zero if $T$ is not a rooted identity tree and 1 otherwise, as the following computation shows: 
In the same way as by Theorem~\ref{T:2} it follows that for $T\in\mbox{MSET}(\CMcal{R})$ and $\vert T\vert=n$,
\begin{align*}
r_T(u)=Z(\mbox{Aut}(T);u,-1,1,\ldots)&=\frac{1}{\vert\mbox{Aut}(T)\vert}
\sum_{\sigma\in\text{\scriptsize Aut}(T)}(-1)^{\sigma_2+\sigma_4+\cdots}u^{\sigma_1}.
\end{align*}
Clearly, if $T\in\CMcal{R}$, then $\vert\mbox{Aut}(T)\vert=1$ and $\sigma_1=n$, thus $r_T(u)=u^n$. It should be noted that if $T$ is not a rooted identity tree, 
{\em i.e.}, $T\not\in\CMcal{R}$ and $\vert T\vert=n$, we have
\begin{align*}
Z(\mbox{Aut}(T);1,-1,1,\ldots)&=\frac{1}{\vert\mbox{Aut}(T)\vert}
\sum_{\sigma\in\text{\scriptsize Aut}(T)}(-1)^{\sigma_2+\sigma_4+\cdots}=0,
\end{align*}
which implies that
\begin{align*}
[z^n]R(z)=\sum_{\substack{T\in\CMcal{R}\\ \vert T\vert=n}}r_T(1)=\sum_{\substack{T\in\text{\scriptsize MSET}(\CMcal{R})\\ \vert T\vert=n}}r_T(1).
\end{align*}
\begin{example}
For $n=3$ we have two \polya trees, namely the chain $T_1$ and the cherry $T_2$. Both belong to the multiset of rooted identity trees. Obviously, $\Aut(T_1) = \{ \text{id} \}$, and $\Aut(T_2) = \{\text{id}, \sigma\}$, where $\sigma$ swaps the two leaves but the root is unchanged. This contributes a minus sign to 
$r_{T_2}(u)$. Thus,
	\begin{align*}
		r_{T_1}(u) &= u^3, &
		r_{T_2}(u) & = \frac{1}{2} (u^3 - u).
	\end{align*}
Note that the cherry $T_2$ is not a rooted identity tree, so $r_{T_2}(1)=0$, while the chain $T_1$ is a rooted identity tree, so $r_{T_1}(u)=u^3$.

For $n=4$ we have the four \polya trees shown in Figure~\ref{F:2}. Except $T_3$, these trees belong to the multiset of rooted identity trees. 
Their automorphism groups are given by $\Aut(T_1) = \Aut(T_2) = \{ \text{id} \}$, and
$$\Aut(T_4) = \{\text{id}, (v_2\,v_3), (v_3\,v_4), (v_2\,v_4), (v_2\, v_3\, v_4), (v_2\,v_4\,v_3) \} \cong S_3.$$
See Figure~\ref{F:0}. This gives
	\begin{align*}
		r_{T_1}(u) &= p_{T_2}(u) = u^4, &
		r_{T_4}(u) &= \frac{1}{6} (u^4 - 3u^2 + 2u).
	\end{align*}
Both $T_1$ and $T_2$ are rooted identity trees, while $T_4$ is not. \hfill{$\diamondsuit$}
\end{example}

In the same way as we got the composition scheme in \eqref{eq:polyadeco}, we can rewrite $R_c(z,u)$ as $R_c(z,u)=C(uzD^*(z))$. 
The expected total weight of all $C$-trees contained in all \polya trees of size $n$, according to
the weights of their decompositions into a $C$-tree and a set of $D^*$-forests, is the $n$-th coefficient of $R_c(z)$, which is
\begin{align*}
	R_c(z) := \left.\frac{\partial}{\partial u} R_c(z,u) \right|_{u=1} = \frac{R(z)}{1 - R(z)} = z + 2z^2 + 4z^3 + 9z^4 + 20 z^5 + 46 z^6 + \cdots.
\end{align*}
By construction, recall \eqref{Rc}, these numbers count the number of points which are fixed
points in all automorphisms of \polya trees that are generated by a root to which rooted identity
trees are attached. For example consider the \polya trees of size $4$ shown in Figure~\ref{F:2}.
The trees $T_1$, $T_2$, and $T_4$ are constructed in this way. In these three trees there are in
total nine points which are always fixed points. Yet, $T_4$ is not a  rooted identity tree. Note that these numbers also count a simple grammar, see OEIS~A$052806$ \cite{Sloane}.


\begin{remark}
It would be desirable to have a similar relation between rooted identity trees and $C$-trees as we have between $C$-trees and \polya trees. However, when setting 
$C(z) = R(z E(z))$ we obtain 
$E(z) = 1+\frac{1}{2}z^2-\frac{1}{3}z^3+\frac{11}{8}z^4-\frac{6}{5}z^5+\frac{629}{144}z^6 + \cdots$, 
a power series with not only nonnegative coefficients. Thus there is no straightforward interpretation in the desired form. 
\end{remark}

\section{Conclusion}\label{S:last}
In this paper we develop a combinatorial framework to describe the relation between \polya trees and simply generating trees. Since we kept the framework light, it is not strong enough to reprove the functional limit theorem presented by Panagiotou and Stufler~\cite{PS}, but it yields a description to this limit theorem which is to our opinion more elementary and more easily accessible to combinatorialists. In addition, we provide not only an alternative proof of the known big-$O$ result on the maximal size of $D$-forests in a random \polya tree, but are able to extend this result. We provide a lower bound of the same order and also precise constants in both bounds. By interpreting all weights on $D$-forests and $C$-trees in terms of automorphisms associated to a \polya tree, we derive the limiting probability that for a random node $v$ the attached $D$-forest $F_n(v)$ is of a given size as well as some structural properties. 

In view of the connection between Boltzmann samplers and generating functions, it comes as no surprise that the ``colored'' Boltzmann sampler from \cite{PS} is closely related to a bivariate generating function. But the unified framework in analyzing the (bivariate) generating functions offers stronger results on the limiting distributions of the size of the $C$-trees and the maximal size of $D$-forests as well as more detailed knowledge on the $D$-forests within a random \polya tree.

\section*{Acknowledgments } \noindent The authors thank the referees for their feedback. 


\end{document}